\documentclass[english,12pt]{amsart}
\usepackage{amsmath,amssymb,amscd,amsfonts,mathrsfs,verbatim}
\usepackage[all]{xy}
\usepackage[mathcal]{euscript}
\usepackage{epsfig}
\usepackage{times}
\usepackage{MnSymbol}

\newcommand{\N}{{\mathbf N}}                   
\newcommand{\R}{{\mathbf R}}                   
\renewcommand{\H}{{\mathbf H}}                   
\newcommand{\C}{{\mathbf C}}                   
\newcommand{\CP}{\mathbf{P}^1}                
\newcommand{\E}{{\mathcal E}}                  
\newcommand{\PP}{{\mathbf P}}                    
\renewcommand{\O}{{\mathcal O}}                 
\newcommand{\Mod}{{\mathcal M}}               
\renewcommand{\d}{\mbox{d}}                      
\newcommand{\gl}{\mathfrak{gl}}                
\newcommand{\End}{{\mathcal E}nd}

\DeclareMathOperator{\res}{res}
\DeclareMathOperator{\Gl}{Gl}

\DeclareMathOperator{\Id}{Id}
\DeclareMathOperator{\coker}{coker}
\DeclareMathOperator{\im}{im}

\DeclareMathOperator{\rank}{rk}
\DeclareMathOperator{\Hilb}{Hilb}
\DeclareMathOperator{\Pic}{Pic}

\DeclareMathOperator{\ad}{ad}
\DeclareMathOperator{\red}{red}
\DeclareMathOperator{\rel}{rel}
\DeclareMathOperator{\Gal}{Gal}
\DeclareMathOperator{\depth}{depth}
\DeclareMathOperator{\Dol}{Dol}
\DeclareMathOperator{\irr}{irr}
\DeclareMathOperator{\GlobExt}{Ext}
\DeclareMathOperator{\tr}{tr}

\newtheorem{prop}{Proposition}[section]

\newtheorem{rk}[prop]{Remark}
\newtheorem{clm}[prop]{Claim}

\newtheorem{thm}[prop]{Theorem}

\newtheorem{construct}{Construction}
\title[The birational geometry of irregular Higgs bundles]
{The birational geometry of irregular Higgs bundles}
\author[Szil\'ard Szab\'o]{Szil\'ard Szab\'o \\ 
  Budapest University of Technology and Economics, \\
  Egry J. u. 1., 1111 Budapest, Hungary \\
  \texttt{szabosz@math.bme.hu}}

\date{\today}

\begin{document}

\begin{abstract}
We give a variant of the Beauville--Narasimhan--Ramanan correspondence 
for irregular parabolic Higgs bundles on smooth projective curves with 
semi-simple irregular part and show that it defines a Poisson isomorphism 
between certain irregular Dolbeault moduli spaces and relative Picard bundles 
of families of ruled surfaces over the curve. 
\end{abstract}

\maketitle

\section{Introduction}

The Beauville--Narasimhan--Ramanan (BNR) correspondence \cite{bnr} provides an equivalence of 
categories between (an open subset of) the category of twisted Higgs bundles $(\E,\theta)$ over a smooth projective curve 
and torsion-free sheaves $S$ of rank $1$ on finite covers of the curve contained in a ruled surface $Z$. 
The functor simply turns the action of the Higgs field into the action of multiplication by a 
variable algebraically independent from the function field of the curve; regularity of the Higgs field 
amounts to saying that we obtain the action of a commutative algebra. 


In the recent years, Higgs bundles with (possibly irregular) singularities have also been extensively studied from 
various perspectives \cite{Biq-Boa}, \cite{Boa1}, \cite{Moc}; 
however, to our knowledge the BNR-construction has not yet been fully worked out. 
As it was pointed out in \cite{A-Sz} by K. Aker and the author, for such a correspondence in the irregular parabolic case 
one needs to perform iterated blow-ups along non-reduced $0$-dimensional subschemes of the fibres of $Z$ over 
the irregular singular points. M. Kontsevich and Y. Soibelman sketch the idea in Section 8.3 of \cite{Kon-Soi} 
(however the role of the parabolic structure on the spectral data is not underlined there). 
Our aim in this paper is to fill out the details of this correspondence. 
In particular, our results imply part 1) of Conjecture 8.6.1 of \cite{Kon-Soi} in the semi-simple case. 
Specifically, we will show the following. 

\begin{thm}
  Given a smooth complex projective curve $C$, a finite set $p_1,\ldots, p_n\in C$ and multiplicities $m_i\in \N$ for 
  $1\leq i \leq n$, there exists a family of ruled surfaces 
$$
  \widetilde{\mathcal{Z}} \to \C^{r\sum_i (m_i+1)} 
$$
over $C$ such that there exists a Poisson isomorphism between dense open subsets of the moduli space of 
stable irregular Higgs bundles on C 
with semi-simple irregular parts at $p_i$ with pole of order $m_i+1$ and no other singularity on the one hand, 
and on the other hand the relative Picard bundle of certain torsion-free sheaves of rank $1$ over a 
relative Hilbert scheme of curves in $\widetilde{\mathcal{Z}}$. 
\end{thm}
In concrete terms, the open subset of the moduli space of stable Higgs bundles appearing in the Theorem is 
characterised by the condition that the spectral curve be smooth. 
For a more precise statement of the result see Theorem \ref{thm:Poisson-isomorphism}. 
On the other hand, an equivalence of categories in the spirit of the above theorem for objects endowed with a 
compatible parabolic structure exists under a milder assumption on the spectral curve, namely integrality; 
see Theorem \ref{thm:parabolicBNR} for a precise statement of this correspondence. 

The paper is organized as follows. In Section \ref{sec:irregular-parabolic-Higgs} we set some notation 
concerning irregular Higgs bundles and parabolic sheaves. 
In Section \ref{sec:correspondence-regular} we give the correspondence in the case of Higgs bundles 
with regular singularities (Proposition \ref{prop:parabolicBNR-logarithmic}). 
Then we turn to the irregular case in Section \ref{sec:correspondence-irregular} and extend 
the correspondence to this case in Theorem \ref{thm:parabolicBNR}. 
Finally, in Section \ref{sec:Poisson} we spell out the natural Poisson structures 
on the relative Dolbeault moduli space and the relative Picard bundle and show that 
the correspondence is a Poisson isomorphism between dense open subsets (Theorem \ref{thm:Poisson-isomorphism}). 

\section{Irregular parabolic Higgs bundles and parabolic sheaves}\label{sec:irregular-parabolic-Higgs}

In this section we will introduce some basic terminology on the one hand concerning irregular Higgs bundles on curves
along the lines of P. Boalch's paper \cite{Boa-survey} and on the other hand concerning parabolic sheaves, 
and establish some elementary results about these objects. 
We end the section with a brief description of a ruled surface. 

\subsection{Irregular parabolic Higgs bundles}

Let $C$ be a smooth complex projective curve over $\C$ and $\E$ a holomorphic vector bundle of rank $r\geq 2$ over $C$. 
We denote by $\O_C, K_C$ the regular and canonical sheaves of $C$ respectively and by $\d$ the differential of meromorphic forms on $C$. 
Let us fix finitely many points $p_1,\ldots, p_n \in C$ and non-negative integers $m_1,\ldots , m_n\in \N$. 
For each $i$ we also fix a local coordinate $z_i$ of $C$ near $p_i$ and an irregular part 
\begin{equation}\label{eq:Qi}
  Q_i = A_i^{m_i} z_i^{-m_i} + \cdots + A_i^1 z_i^{-1}
\end{equation}
where $A_i^j\in \gl_r(\C)$ are semi-simple endomorphisms satisfying 
$$
  [A_i^j,A_i^{j'}] = 0. 
$$
In (\ref{eq:Qi}) the exponents of $z_i$ are meant to increase from left to right; 
in the particular case $m_i = 0$ the irregular part is $Q_i = 0$, and in this case 
we say that $\theta$ has regular or logarithmic singularities. We set 
$$
  D = (m_1 +1) \cdot p_1+ \cdots + (m_n +1) \cdot p_n
$$
as an effective divisor on $C$ with associated reduced divisor  
$$
  D_{\red} = p_1+ \cdots + p_n.
$$
Moreover we set 
\begin{equation}\label{eq:L}
 	L = K_{C} (D). 
\end{equation}
By (untwisted) irregular Higgs field with local form (\ref{eq:Qi}) we mean a section 
$$
  \theta \in H^0(C, \End (\E)\otimes_{\O_C} L)
$$
such that for any $i\in\{ 1,\ldots , n\}$ there exists a local trivialisation of $\E$ near $p_i$ with respect to which we have an expansion 
\begin{equation}\label{eq:theta-local}
  \theta = \left( \d Q_i + \Lambda_i z_i^{-1} + \mbox{ holomorphic terms} \right) \d z_i
\end{equation}
for some $\Lambda_i\in \gl_r(\C)$. A quasi-parabolic structure compatible with $\theta$ at $p_i$ is a filtration 
\begin{equation}\label{eq:parfiltr}
	\{ 0\} \subset F_i^{l_i-1} \subset \cdots \subset F_i^1 \subset F_i^0 = \E_{|p_i}
\end{equation}
of some length $1\leq l_i\leq r$ preserved by the matrix $\Lambda_i$ and the irregular part $Q_i$ (i.e. by all the matrices 
$A_i^1,\ldots , A_i^{m_i}$). 
\begin{clm}\label{clm:A-preserves-F}
 $Q_i$ preserves the filtration (\ref{eq:parfiltr}) if and only if for any $l$ the piece $F_0^l$ in 
  (\ref{eq:parfiltr}) is the direct sum of its intersections with the various simultaneous eigenspaces of 
  $A_i^1,\ldots , A_i^{m_i}$. 
\end{clm}
\begin{proof}
The direction $\Leftarrow$ is trivial. For the converse, we merely need to show it for a single $A_i^j$ 
that we denote by $A$ for ease of notation. 
If for instance 
$$
  \lambda_1 v_1 + \lambda_2 v_2 \in F_0^l
$$
for some $\lambda_j\neq 0$ and $v_j$ in the $\zeta_j$-eigenspace of $A$ with $\zeta_1\neq \zeta_2$ then 
$$
  A(\lambda_1 v_1 + \lambda_2 v_2) = \zeta_1(\lambda_1 v_1 + \lambda_2 v_2) + (\zeta_2 - \zeta_1)\lambda_2 v_2.
$$
Now as by assumption the left-hand side and the first term on the right-hand side belong to $F_0^l$, the 
same thing follows for the second term on the right-hand side, and thus (as $\zeta_2 - \zeta_1\neq 0$) for $\lambda_2 v_2$ too, 
which in turn implies the same thing for $\lambda_1 v_1$ as well. 
The same kind of argument applies for a vector with components in more than just two different eigenspaces.
\end{proof}
A compatible parabolic structure at $p_i$ is the datum of a compatible quasi-parabolic structure at $p_i$ 
and parabolic weights 
\begin{equation}\label{eq:parweights}
	1> \alpha_i^{l_i-1} > \cdots > \alpha_i^0 \geq 0. 
\end{equation}

\subsection{Parabolic sheaves}

It is convenient to recall the notion of an $\R$-parabolic sheaf on a complex manifold $X$ 
with divisor $D_{\red}$, a reduced effective divisor on $X$: 
this is a coherent sheaf $S$ with a decreasing filtration $S_{\bullet}$ indexed by $\R$ so that 
for all $\alpha\in\R$ 
\begin{enumerate}
 \item there exists some $\varepsilon >0$ with $S_{\alpha - \varepsilon} = S_{\alpha}$ and \label{R-par-sheaf-1}
 \item we have $S_{\alpha + 1} = S_{\alpha} \otimes \O_X(-D_{\red})$. \label{R-par-sheaf-2}
\end{enumerate}
\begin{clm} 
The categories of $\R$-parabolic locally free sheaves on a curve $C$ with divisor 
$$
  D_{\red} =  p_1 + \cdots + p_n
$$
and of parabolic bundles on $C$ with divisor $D_{\red}$ are isomorphic.
\end{clm}
\begin{proof} 
To a parabolic bundle $\E$ with filtration (\ref{eq:parfiltr}) associate the $\R$-parabolic sheaf 
$\E_{\bullet}$ defined as follows. Near the generic point $z\notin D_{\red}$ for every $\alpha\in\R$ we let $\E_{\alpha} = \E$. 
For $i\in\{ 1, \ldots, n \}$ and $\alpha_i^{l-1} < \alpha \leq \alpha_i^{l}$ we define $\E_{\alpha}$ 
in a small neighborhood of $p_i\in D$ not containing any other $p_{i'}$ as the kernel of the composition map 
$$
  \pi_i^l: \E \xrightarrow{eval_{p_i}} \E|_{p_i} = F_i^0 \to F_i^0/F_i^{l}. 
$$
For $\alpha\notin [0,1]$ we extend the definition by property (\ref{R-par-sheaf-2}) above. 

To an $\R$-parabolic sheaf we associate the vector bundle whose local sections are given by the sheaf $\E_0$, 
with filtration (\ref{eq:parfiltr}) defined as follows: for any vector $v\in \E_{|p_i}$ we let 
$v\in F_i^{l}$ if and only if any local section of $\E_0$ extending $v$ (i.e., whose specialization at 
$p_i$ is $v$) is actually a section of $\E_{\alpha_i^l}$.

These two constructions are clearly inverse to each other. 
\end{proof}

\subsection{A ruled surface}

Finally, let us introduce some notation concerning ruled surfaces. Let us denote by $Z$ the surface 
$$
	Z = \PP_{C}(\O_{C} \oplus L^{\vee}), 
$$
the fiberwise projectivisation of the total space of the line bundle $L^{\vee}$ dual to $L$. 
There exist a natural projection
$$
	p: Z \to C, 
$$
a relative hyperplane bundle $\O_Z(1)$ and canonical sections 
\begin{equation}\label{eq:canonical-sections}
 	\xi\in H^0(Z,\O_Z(1)), \zeta\in H^0(Z, p^*L\otimes \O_Z(1)).
\end{equation}
The divisors 
$$
  (\xi = 0), (\zeta = 0)
$$
are called the divisor at infinity and the $0$-divisor respectively. The line bundle $\O_Z(1)$ is trivial 
on the complement of the divisor at infinity, and the same holds for the $0$-divisor.

\section{The correspondence in the regular case}\label{sec:correspondence-regular}

Now let $(\E,\theta)$ be a Higgs bundle on a smooth projective curve $C$ and let us assume that all the singularities of 
$\theta$ are regular (i.e. $Q_i=0$ in (\ref{eq:Qi})). 
Then for all $i,l$ the restriction of $\theta$ induces a map 
$$
  \E_{\alpha_i^{l}} \to \E \otimes L,
$$
and since $\theta$ is assumed to be compatible with the parabolic structure, the composition 
$$
  \E_{\alpha_i^{l}} \to \E \otimes L \xrightarrow{\pi_i^l\otimes Id_L}  F_i^0/F_i^l \otimes L|_{p_i}
$$
is the $0$-morphism. We infer that the restriction of $\theta$ to $\E_{\alpha_i^{l}}$ factors through 
$\E_{\alpha_i^{l}} \otimes L$, and hence gives rise to maps
$$
  \theta_{\alpha}:\E_{\alpha} \to \E_{\alpha} \otimes L
$$
for every $\alpha$. To emphasize that we are dealing with the regular case set 
$$
  Z_{\red} = \PP_{C}(\O_{C} \oplus K_C(D_{\red})^{\vee}) 
$$
and for ease of notation let us denote by $p,\zeta,\xi$ the corresponding projection from $Z_{\red}$ to $C$ and canonical sections. 
It follows from what we have said above that we can define the spectral sheaves corresponding to the subsheaves $\E_{\alpha}$: 
\begin{equation}
 S_{\E_{\alpha}}  = \coker (p^*\E_{\alpha} \xrightarrow{\xi\otimes  p^*\theta+\zeta} p^*(\E_{\alpha}\otimes K_C(D_{\red}))\otimes \O_{Z_{\red}}(1)). \label{eq:spectral-sheaf-E-alpha}
\end{equation}
This is a coherent sheaf whose support is denoted 
\begin{equation}\label{eq:spectral-curve}
  \Sigma = (\det (\xi\otimes  p^*\theta + \zeta) )
\end{equation}
and called spectral curve. (In principle, $\Sigma$ could depend on $\alpha$, 
but the proof of the proposition below shows in particular that this is not the case.) 
\begin{prop}\label{prop:parabolicBNR-logarithmic}
\begin{enumerate}
 \item If $(\E,\theta)$ is a logarithmic Higgs bundle with compatible parabolic structure at $D_{\red}$ on a curve 
$C$ then the sheaves $S_{\E_{\alpha}}$ are pure of dimension $1$ and form an $\R$-parabolic sheaf on $Z_{\red}$ 
with divisor $p^{-1}(D_{\red})$. 
\item Conversely, given any $\R$-parabolic sheaf $S_{\alpha}$ with divisor $p^{-1}(D_{\red})$
and sheaves pure of dimension $1$ on $Z_{\red}$ with support disjoint from $\xi = 0$, then the sheaves 
$\E_{\alpha} = p_* S_{\alpha}\otimes K_C(D_{\red})^{\vee}$ 
form an $\R$-parabolic bundle on $C$ and $\theta$ given by 
\begin{equation}\label{eq:original-BNR}
 \theta = p_* (-\zeta \cdot \, : S_{\alpha} \otimes K_C(D_{\red})^{\vee} \to S_{\alpha})
\end{equation}
defines a Higgs field with regular singularities and compatible with the parabolic structure. 
\end{enumerate}
The two constructions are quasi-inverse to each other. 
\end{prop}
\begin{proof}
By \cite{bnr}, we have 
$$
  p_* S_{\E_{\alpha}}  = \E_{\alpha}\otimes L. 
$$
If the schematic support of $S_{\E_{\alpha}}$ had a $0$-dimensional embedded component, then the 
same would hold for $\E_{\alpha}\otimes L$; this however is impossible as $\E_{\alpha}$ is locally 
free, in particular torsion-free. 

For any $\alpha < \beta$ the quotient $\E_{\alpha}/\E_{\beta}$ is a skyscaper sheaf 
supported at $D_{\red}$; let us denote this sheaf by $\C^{\alpha,\beta}$. 
Now since $p^*\C^{\alpha,\beta}$ is pure of dimension $1$ supported on $p^{-1}(D_{\red})$ 
and the map $p^*\theta + \zeta$ is non-zero on this sheaf (unless it is itself the $0$-sheaf), 
its kernel must vanish and we have the diagram 
$$
	\xymatrix{
			& 0 & 0 & 0 & \\
		0 \ar[r] & p^*\C^{\alpha,\beta} \ar[r] \ar[u]& p^*(\C^{\alpha,\beta}\otimes L)\otimes\O_{Z_{\red}}(1) \ar[r] \ar[u] & 
												cK^{\alpha,\beta} \ar[r] \ar[u] & 0 \\
		0 \ar[r] & p^*\E_{\alpha} \ar[r]^{\xi\otimes p^*\theta + \zeta}\ar[u] & p^*(\E_{\alpha}\otimes L)\otimes \O_{Z_{\red}}(1) \ar[r]\ar[u] & S_{\E_{\alpha}}\ar[r] \ar[u] & 0 \\
		0 \ar[r] & p^*\E_{\beta}\ar[r]^{\xi\otimes p^*\theta + \zeta}\ar[u] & p^*(\E_{\beta}\otimes L)\otimes \O_{Z_{\red}}(1) \ar[r]\ar[u] & S_{\E_{\beta}}\ar[r] \ar[u] & 0 \\
					& 0 \ar[u] & 0 \ar[u] & 0 \ar[u] & 
	}
$$
which defines the cokernel sheaf $cK^{\alpha,\beta}$. The top row shows that $cK^{\alpha,\beta}$ is 
a torsion sheaf supported in dimension $0$ over 
$$
  \Sigma \cap p^{-1}(D_{\red})
$$
so $S_{\E_{\alpha}}$ is an upper elementary modification of $S_{\E_{\beta}}$ at these points. 
This shows that the sequence of sheaves $S_{\E_{\bullet}}$ is decreasing.  
Now for any $\alpha$ there exists $\varepsilon >0$ such that we have $\E_{\alpha - \varepsilon} = \E_{\alpha}$; 
it then follows that $S_{\E_{\alpha - \varepsilon}} = S_{\E_{\alpha}}$ too. 
Finally, as $\E_{\alpha +1} = \E_{\alpha}(-(z_0 + \cdots + z_n))$ the projection formula 
\begin{equation}\label{eq:projection-formula}
  p^*(\O_C(D_{\red})) = \O_{Z_{\red}}(p^{-1}(D_{\red})), 
\end{equation}
implies 
$$
  S_{\E_{\alpha +1}} = S_{\E_{\alpha}}(-p^{-1}(D_{\red})).
$$

For the converse direction, all the sheaves $\E_{\alpha}$ are locally free as torsion-free sheaves on 
smooth curves are locally free. 
As multiplication by $(- \zeta)$ preserves $S_{\alpha}$ it follows that $\theta$ preserves $\E_{\alpha}$, 
which in turn means that $\theta$ is compatible with the parabolic filtration. 
Finally, the fact that $\E_{\alpha}$ is $\R$-parabolic follows again from (\ref{eq:projection-formula}). 

The  fact that the two constructions are inverse to each other is straightforward \cite{bnr}: 
a $\pi_*\O_{\Sigma}$-module structure on $\E$ is the same thing as an $\O_C$-linear 
endomorphism $\theta$, and over the locus where the eigenvalues of $\theta$ are all distinct 
the $\pi_*\O_{\Sigma}$-module structure simply gives the decomposition of $\E$ into 
its eigenspaces for $\theta$. 
\end{proof}

\section{The correspondence in the irregular case}\label{sec:correspondence-irregular}

Proposition \ref{prop:parabolicBNR-logarithmic} gives a satisfactory result at the regular singularities; however, at the irregular 
singularities  the sheaves must possess a finer structure in order for such a correspondence to hold. 
This structure may be conveniently expressed using the technique of proper transform 
of $\R$-parabolic sheaves defined in Section 5 of \cite{A-Sz}. 
Let us now describe this construction. 
To fix our ideas, let us first treat the case $m_i=1$, and set $A_i = A_i^{1}$. 
Fix a local chart $z_i$ vanishing at $p_i$; this then gives rise to a local trivialization 
\begin{equation}\label{eq:2nd-order-lambda}
 \lambda_i = z_i^{-2}\d z_i 
\end{equation} 
of the line bundle (\ref{eq:L}). The assumption $m_i=1$ means that $\theta$ has local form 
\begin{equation}\label{eq:2nd-order-theta}
  \theta = (-A_i + z_i \Lambda_i + \cdots ) \lambda_i
\end{equation} 
where the dots stand for higher order terms in $z_i$. 
Let us introduce the local holomorphic function $\zeta_i$ on $Z\setminus (\xi )$ by 
\begin{equation}\label{eq:zeta-i}
  \zeta = \zeta_i p^* \lambda_i; 
\end{equation}
as the left hand side is a section of $p^* L$ and $p^*\lambda_i$ is a local section of $L$, 
$\zeta_i$ is then a local function. In what follows, we will drop $p^*$ from the notation. 
Let us denote the distinct eigenvalues of $-A_i$ by $\zeta_{i,1},\ldots,\zeta_{i,L_i}$, each repeated 
with a certain multiplicity $d_{i,1},\ldots, d_{i,L_i}$. 
It is known that the residue $\Lambda_i$ may be assumed up to a gauge transformation to be block diagonal with 
respect to the eigenspace-decomposition of $A_i$ (actually, the same thing holds for any prescribed 
number of the lower order terms). Let us denote the block of $\Lambda_i$ corresponding to the eigenvalue 
$\zeta_{i,l}$ of $-A_i$ by $\Lambda_{i,l}$. The spectral curve $\Sigma$ passes through the points 
$(p_i,\zeta_{i,1}),\ldots, (p_i,\zeta_{i,L_i})$ (where in the second coordinate $\zeta_{i,l}$ means the point 
$(\zeta_{i,l}:1)$ in the projective chart of the fiber $\E|_{p_i}$ given by the sections $\zeta,\xi$). 
Let us consider the formal splitting of $\E$ into the generalised eigenspaces for $\theta$ near $p_i$: 
\begin{equation}\label{eq:formal-decomposition}
  \E = \E^{i,1} \oplus \cdots \oplus \E^{i,r}; 
\end{equation}
this splitting holds over the field of formal Puiseux-series in $z_i$.
By Kneser's lemma, the eigenvalues $\zeta_{i,j}$ of $\theta$ are of the form 
\begin{equation}\label{eq:2nd-order-theta-eigenvalues}
   \left( \frac{\zeta_{i,j}^{2}}{z_i^{2}} + O\left( z_i^{-2+\frac 1r}\right) \right) \d z_i
\end{equation}
where for $1\leq j \leq r$ we denote by $\zeta_{i,j}$ the $j$'th eigenvalue of $-A_i$ in the compatible basis of (\ref{eq:2nd-order-theta}). 
Let us gather the generalised eigenspaces appearing in (\ref{eq:formal-decomposition}) corresponding to equal values of $\zeta_{i,j}$ 
to define a coarser decomposition 
\begin{equation}\label{eq:convergent-decomposition}
  \E = \tilde{\E}^{i,1} \oplus \cdots \oplus \tilde{\E}^{i,L_i}; 
\end{equation}
for example, 
$$
  \tilde{\E}^{i,1} = \bigoplus_j \E^{i,j} 
$$
the summation ranging for all values of $j$ such that $\zeta_{i,j}$ is equal to a fixed complex number, etc.  
\begin{clm}\label{clm:regular-splitting}
 The decomposition (\ref{eq:convergent-decomposition}) is defined over the field $\C\{ z_i \}[z_i^{-1}]$ 
 of convergent Laurent series in the variable $z_i$, and the summands are locally free $\C\{ z_i \}$-modules. 
 Furthermore, it is preserved by $\theta$.
\end{clm}
\begin{proof}
 Let $L$ be the splitting field of the characteristic polynomial 
 $$
   \chi_{\theta(\partial_{z_i})}(\zeta,z_i) = \det(\theta (\partial_{z_i}) + \zeta )
 $$
 of $\theta(\partial_{z_i})$ over the ring of formal power series $\C[[z_i]]$. 
 Then the action of $\Gal(L|\C[[z_i]])$ maps a generalised eigenvector with Puiseux-expansion having highest-order term $\zeta_{i,j}$ 
 to a generalised eigenvector with equal highest-order term. 
 Therefore, the direct sum of such generalised eigenspaces is $\Gal(L|\C[[z_i]])$-invariant. 
 On the other hand, the solution curve of $\chi_{\theta(\partial_{z_i})}$ is the graph of a multi-valued analytic function in $z_i$, 
 hence the generalised eigenvectors are given by convergent Puiseux-series. As $\E$ is torsion-free, it may only have torsion-free submodules. 
 This proves the first two statements. 
 
 As for the third statement, notice that as the factors $\E^{i,l}$ are direct sums of generalised $\theta$-eigenspaces, 
 they are preserved by $\theta$ over the generic point. So for any $l\neq l'$ the map 
 $$
  \E^{i,l} \to \E^{i,l'}
 $$
 induced by $\theta$ is generically $0$. We conclude by the second statement. 
\end{proof}

To simplify notation, from now on we will write $\E^{i,l}$ for $\tilde{\E}^{i,l}$. 
Let $\theta^{i,l}$ stand for the restriction of $\theta$ to $\E^{i,l}$. 
The claim implies that we have a local direct sum decomposition
$$
  (\E,\theta ) = (\E^{i,1},\theta^{i,1}) \oplus \cdots \oplus (\E^{i,L_i}, \theta^{i,L_i})
$$
with $\E^{i,l}|_{p_i}$ equal to the $\zeta_{i,l}$-eigenspace of $A$. 
Furthermore, by Claim \ref{clm:A-preserves-F}, the above splitting is also compatible with the parabolic structure 
in the sense that for any $\alpha$ we also have 
$$
  \E_{\alpha} = \E_{\alpha}^{i,1} \oplus \cdots \oplus \E_{\alpha}^{i,L_i}, 
$$
where 
$$
  \E_{\alpha}^{i,l} = \E_{\alpha} \cap \E^{i,l}. 
$$
Define the affine curves $\Sigma^{i,l}$ by the vanishing of the sections
$$
  \det(\zeta \mbox{I}_{\E^{i,l}} + \xi \theta^{i,l}) 
$$
where $\mbox{I}_{\E^{i,l}}$ stands for the identity transformation of $\E^{i,l}$. 
Let us now blow up $Z$ at the points $(p_i,\zeta_{i,l})$ for $l\in \{1,\ldots ,L_i\}$: 
\begin{equation}\label{eq:blow-up-i}
  \sigma_i: Z^i \longrightarrow Z. 
\end{equation}
The exceptional divisor of $\sigma_i$ corresponding to $(p_i,\zeta_{i,l})$ is denoted $E_{i,l}$. 
Specifically, introduce homogeneous coordinates 
\begin{equation}
  (z_{i,l}':\zeta_{i,l}') \in E_{i,l} = \CP; 
\end{equation}
and let $U_{i,l}$ denote a neighborhood of $(p_i,\zeta_{i,l})$ containing no other point $(p_{i'},\zeta_{i',l'})$ for $i'\neq i, l'\neq l$. 
Then the preimage by $\sigma_i$ of $U_{i,l}$ is given by 
\begin{equation}\label{eq:blow-up-coordinates}
  \sigma_i^{-1}(U_{i,l}) =\{ z_i \zeta_{i,l}' = z_{i,l}'(\zeta_i - \zeta_{i,l}) \} \subset U_{i,l} \times \CP. 
\end{equation}
Let us denote by 
\begin{equation}\label{eq:Uil'}
  U_{i,l}'\subset \sigma_i^{-1}(U_{i,l})
\end{equation}
the affine subset defined by $z_{i,l}'\neq 0$. On $U_{i,l}'$ we then have 
$$
  \zeta_i - \zeta_{i,l} = z_i \frac{\zeta_{i,l}'}{z_{i,l}'}; 
$$
in particular, on $U_{i,l}'$ the equality $z_i=0$ implies $\zeta_i = \zeta_{i,l}$. 
Here $\zeta_{i,l}',z_{i,l}'$ are to be understood as sections of $\O_{Z^i}(-E_{i,l})$, since 
$$
  \O_{Z^i}(-E_{i,l})|_{E_{i,l}}\cong \O_{E_{i,l}}(1).
$$
The equation of the total transform $\sigma_i^{-1}(\Sigma)$ of the spectral curve in the above affine chart therefore writes 
\begin{equation*}
  \det(\zeta \mbox{I}_{\E^{i,l}} - \theta^{i,l}) =  
  \det \left( z_i \frac{\zeta_{i,l}'}{z_{i,l}'} \lambda_i \mbox{I}_{\E^{i,l}}  - \tilde{\theta}^{i,l} \right) 
\end{equation*}
with 
\begin{equation}\label{eq:theta-tilde}
 \tilde{\theta}^{i,l} = \theta^{i,l} - \zeta_{i,l}\lambda_i  \mbox{I}_{\E^{i,l}} = (B_{i,l} + O(z_i)) \frac{\d z_i}{z_i} = (B_{i,l} + O(z_i)) z_i \lambda_i
\end{equation}
as $z_i\to 0$. 
We now resolve the quotient in the determinant above by writing  
\begin{equation}\label{eq:total-transformed-spectral-curve}
  \det(\zeta \mbox{I}_{\E^{i,l}} - \theta^{i,l}) =  
  \left( \frac 1{z_{i,l}'} \right)^{d_{i,l}} \det \left( z_i \zeta_{i,l}' \lambda_i \mbox{I}_{\E^{i,l}}  - {z_{i,l}'} \tilde{\theta}^{i,l} \right).
\end{equation}
This is therefore the defining relation of the total transform $\sigma_i^{-1}(\Sigma)$; observe now that if 
$$
  z_i = 0 = \zeta_i - \zeta_{i,l}
$$
then the matrix in the argument of the right hand side of (\ref{eq:total-transformed-spectral-curve}) vanishes for any $(z_{i,l}':\zeta_{i,l}')$. 
Said differently, $E_{i,l}$ is a component of multiplicity $d_{i,l}$ in $\sigma_i^{-1}(\Sigma)$. 
The proper transform $\Sigma^i$ of $\Sigma$ with respect to $\sigma_i$ is therefore given in the chart $U_{i,l}'$ by the formula 
\begin{equation}\label{eq:proper-transformed-spectral-curve}
  \det(\zeta_{i,l}'  \lambda_i \mbox{I}_{\E^{i,l}} -  {z_{i,l}'}z_i^{-1}\tilde{\theta}^{i,l}) \in H^0(U_{i,l}', (p\circ \sigma_i )^*(L\otimes \End(\E^{i,l})) \otimes \O_{Z^i}(-E_{i,l}))
\end{equation}
(compare with (\ref{eq:canonical-sections})). Notice that here 
$$
  z_i^{-1}\tilde{\theta}^{i,l} = (B_{i,l} + O(z_i)) \lambda_i
$$
is a regular section of $(p \circ \sigma_i )^*(L\otimes \End(\E^{i,l}))$. 
In particular, we see that (\ref{eq:proper-transformed-spectral-curve}) has precisely the same form as the characteristic equation 
(\ref{eq:spectral-curve}) of a logarithmic Higgs field with residue $B_{i,l}$ with respect to the auxiliary 
variable $(\zeta_{i,l}'\lambda_i : z_{i,l}')$ instead of $(\zeta:\xi)$. 
Just as in the regular case, from the assumption that $B_{i,l}$ preserves the parabolic filtration we deduce that $z_i^{-1}\tilde{\theta}^{i,l}$ restricts to maps 
\begin{equation*}
  z_i^{-1}\tilde{\theta}_{\alpha}^{i,l}:(p \circ\sigma_i)^* \E_{\alpha}^{i,l} \rightarrow (p\circ\sigma_i)^*(\E_{\alpha}^{i,l} \otimes L) 
\end{equation*}
of locally free sheaves over $U_{i,l}'$; or equivalently, maps 
\begin{equation}\label{eq:second-order-parabolic-map}
  z_i^{-1}\tilde{\theta}_{\alpha}^{i,l}:(p \circ\sigma_i)^* (\E_{\alpha}^{i,l} \otimes L^{\vee}) \rightarrow (p\circ\sigma_i)^*\E_{\alpha}^{i,l}.
\end{equation}
Now the proper transform $\Sigma^i$ given by equation (\ref{eq:proper-transformed-spectral-curve}) 
has an obvious refinement in terms of the spectral sheaves $S_{\E_{\alpha}}$ too. 
Namely, analogously to (\ref{eq:spectral-sheaf-E-alpha}) we may define coherent sheaves on $Z^i$ by the formulae 
\begin{equation}\label{eq:parabolic-spectral-sheaf}
   S_{\E_{\alpha}}^{i} = \coker(\zeta_{i,l}' \lambda_i \mbox{I}_{(\E_{\alpha}^{i,l}\otimes L^{\vee})} -  {z_{i,l}'}z_i^{-1}\tilde{\theta}_{\alpha}^{i,l}) 
\end{equation}
locally on the affine charts $U_{i,l}'$ and as $S_{\E_{\alpha}}$ away from these charts. 
The support of $S_{\E_{\alpha}}^{i}$ is then obviously equal to $\Sigma^i$. 
\begin{clm}\label{clm:R-parabolic}
 Let $z_i\in U\subset C$ be an open set such that $z_{i'}\notin U$ for $i'\neq i$. 
 Then on the open subset $(p \circ\sigma_i)^{-1} U \subset Z^i$ we have 
 $$
  (p \circ\sigma_i)_* S_{\E_{\alpha}}^{i} = \E_{\alpha}.
 $$
 The sheaves $S_{\E_{\bullet}}^{i}$ are pure and define an $\R$-parabolic sheaf with divisor 
\begin{equation}\label{eq:parabolic-divisor}
   (p \circ \sigma_i)^{-1} (p_i). 
\end{equation}
\end{clm}
\begin{proof}
 Let us first show purity: this is a simple consequence of the Auslander--Buchsbaum formula \cite{Huy-Lehn} pp. 4--5: using that 
 $Z^i$ is regular of dimension $2$ at any point $x\in \Sigma^i$ and the projective dimension of $(S_{\E_{\alpha}}^{i})_x$ 
 is by definition equal to $1$ we get that $\depth((S_{\E_{\alpha}}^i)_x)=1$, so its schematic support does not 
 have a $0$-dimensional embedded component. 
 
 By definition for any $V\subseteq U$ we have 
 $$
  (p \circ\sigma_i)_* S_{\E_{\alpha}}^{i} (V) =  S_{\E_{\alpha}}^{i} ((p \circ\sigma_i)^{-1}(V)).
 $$
 As the support of $S_{\E_{\alpha}}^{i}$ is contained in the disjoint union of the open sets (\ref{eq:Uil'})for $l\in\{ 1,\ldots , L_i \}$, 
 the right-hand side is a direct sum of the Abelian groups 
 $$
  S_{\E_{\alpha}}^{i} ((p \circ\sigma_i)^{-1}(V)\cap U_{i,l}'). 
 $$
 This latter is in turn obtained by the sheafification of the pre-sheaf 
 $$
  ((p \circ\sigma_i)^*\E_{\alpha}^{i,l} / \im(\zeta_{i,l}' \lambda_i \mbox{I}_{(\E_{\alpha}^{i,l}\otimes L^{\vee})} -  
  {z_{i,l}'}z_i^{-1}\tilde{\theta}_{\alpha}^{i,l}))((p \circ\sigma_i)^{-1}(V)\cap U_{i,l}'). 
 $$
 Mapping a section of $\E_{\alpha}^{i,l}$ on $V$ to the class of its pull-back thus gives us a canonical epimorphism of sheaves of vector-spaces 
 $$
  \iota : \oplus_{l=1}^{L_i} \E_{\alpha}^{i,l} \to (p \circ\sigma_i)_* S_{\E_{\alpha}}^{i}. 
 $$
 We now show that it is a monomorphism. Notice that both the source and target of $\iota$ are torsion-free sheaves 
 on  $C$: for $\E_{\alpha}^{i,l}$ this holds as it is a subsheaf of the locally free sheaf $\E$, and for 
 $(p \circ\sigma_i)_* S_{\E_{\alpha}}^{i}$ it follows from purity of $S_{\E_{\alpha}}^{i}$ since the 
 support $\Sigma^i$ of this latter sheaf is finite over $C$. 
 Therefore, it is sufficient to show that $\iota$ is generically a monomorphism. 
 The support of $S_{\E_{\alpha}}^{i}$ is equal to $\Sigma^i \cap (p \circ\sigma_i)^{-1}(U)$; 
 in particular by the assumption that $\Sigma^i$ is reduced, it is $r$ to $1$ over $U$. 
 On the open subset $U^0\subseteq U$ over which the $r$ branches of $\Sigma^i$ are distinct the map $\iota$ 
 corresponds to writing sections of $\E_{\alpha}^{i,l}$ as $\O_U$-linear 
 combinations of elements of a basis given by eigensections of $\theta$.  
 As such a linear combination determines the section uniquely, we get the injectivity of $\iota$ over $U^0$. 
 
 The statement that $S_{\E_{\alpha}}^{i}$ form an $\R$-parabolic sheaf follows easily from the projection formula 
 $$
  \O_{(p \circ\sigma_i)^{-1}(U)}((p \circ\sigma_i)^{-1}(p_i)) = (p \circ\sigma_i)^*(\O_U(p_i)). 
 $$
\end{proof}

We are now going to extend the above construction recursively to the case of arbitrarily high order poles $p_i$. 
Denote by $-\zeta_{i,j}^{m}/m$ the $j$'th eigenvalue of the matrix $A_i^m$ appearing in (\ref{eq:Qi}). 
Again by Kneser's lemma, the eigenvalues $\zeta_{i,j}$ of $\theta$ are of the form 
\begin{equation}\label{eq:theta-eigenvalues}
   \left( \frac{\zeta_{i,j}^{m_i}}{z_i^{m_i+1}} + \cdots + \frac{\zeta_{i,j}^{1}}{z_i^2} + O \left( z_i^{-2+\frac 1r}\right) \right) \d z_i
\end{equation}
near $p_i$. Correspondingly, there exists a local splitting (\ref{eq:convergent-decomposition}) of $\E$ where each direct summand 
consists of the direct sum of all eigenspaces of $\theta$ having the same expansion up to order $z_i^{-2}$. 
Just as in Claim \ref{clm:regular-splitting}, this decomposition is actually defined over $\C(\{ z_i\})$. 
We now construct a birational map 
\begin{equation}\label{eq:master-surface}
  \tilde{\sigma}: \widetilde{Z} \to Z 
\end{equation}
recursively. 
If $m_i>1$ then apply a blow-up 
$$
  \sigma_{i,j} :Z^{i,j} \longrightarrow Z 
$$
at each point $(p_i, \zeta_{i,j}^{m_i})\in Z$ with respect to the local trivialisation 
\begin{equation}\label{eq:general-lambda}
 \lambda_i = z_i^{-m_i-1}\d z_i 
\end{equation} 
of $L$ near $p_i$. For each eigenvalue $\zeta_{i,j}^{m_i}$ the point $(p_i, \zeta_{i,j}^{m_i})\in Z$ is 
only taken once, independently of its multiplicity as an eigenvalue. 
In concrete terms, using the holomorphic coordinate (\ref{eq:zeta-i}) introduce new homogeneous coordinates 
\begin{equation}
  (z_{i,j}':\zeta_{i,j}') \in E_{i,j} = \CP; 
\end{equation}
and let $U_{i,j}$ denote a neighborhood of $(p_i,\zeta_{i,j}^{m_i})$ containing no other point $(p_{i'},\zeta_{i',j'}^{m_{i'}})$ for $i'\neq i, j'\neq j$. 
Then the preimage by $\sigma_{i,j}$ of $U_{i,j}$ is given by 
\begin{equation}\label{eq:blow-up-coordinates-2}
  \sigma_{i,j}^{-1}(U_{i,j}) =\{ z_i \zeta_{i,j}' = z_{i,j}'(\zeta_i - \zeta_{i,j}^{m_i}) \} \subset U_{i,j} \times \CP. 
\end{equation}
Let us define an affine open chart of $Z^{i,j}$ by 
$$
  U_{i,j}' = (z_{i,j}'\neq 0). 
$$
Now just as in (\ref{eq:total-transformed-spectral-curve}) if we denote by $\theta^{i,j}$ the 
restriction of $\theta$ to the $-\zeta_{i,j}^{m_i}/m_i$-eigenspace $\E^{i,j}$ of $A_i^{m_i}$, 
then the equation of the total transform of $\Sigma$ with respect to $\sigma_{i,j}$ in $U_{i,j}'$ writes 
\begin{equation}\label{eq:total-tranformed-spectral-curve-m}
  \left( \frac 1{z_{i,j}'} \right)^{d_{i,j}} \det \left( z_i \zeta_{i,j}' \lambda_i \mbox{I}_{\E^{i,j}}  - {z_{i,j}'} \tilde{\theta}^{i,j} \right)
\end{equation}
where $d_{i,j}$ stands for the multiplicity of $\zeta_{i,j}^{m_i}$ as an eigenvalue of $-m_i A_i^{m_i}$ and 
\begin{align}
  \tilde{\theta}^{i,j} & \in (p \circ \sigma_{i,j})^* (L (-p_i)  \otimes \End( \E^{i,j})) \notag \\
  \tilde{\theta}^{i,j} & = \left( \d \tilde{Q}_{i,j} + \Lambda_{i,j} (z_i)^{-1} + \mbox{ holomorphic terms} \right) \d z_i
  \label{eq:theta-tilde-local}
\end{align}
for some 
$$
  \tilde{Q}_{i,j} = \tilde{A}_{i,j}^{m_i-1} z_i^{-m_i-1} + \cdots + \tilde{A}_{i,j}^1 z_i^{-1}. 
$$
Specifically, $\tilde{A}_{i,j}^m$ is the restriction of $A_{i,j}^m$ to the $\zeta_{i,j}^{m_i}$-eigenspace of $-m_i A_{i,j}^{m_i}$. 
Again, $E_{i,j}$ is a component of multiplicity $d_{i,j}$ of the total transform of $\Sigma$ by the transformation 
$\sigma_{i,j}$, and the proper transform is given by 
$$
  \det(\zeta_{i,j}'  \lambda_i \mbox{I}_{\E^{i,j}} -  {z_{i,j}'}z_i^{-1}\tilde{\theta}^{i,j}) \in H^0(U_{i,j}', (p\circ \sigma_{i,j} )^*(L\otimes \End(\E^{i,j})) \otimes \O_{Z^{i,j}}(-E_{i,j})). 
$$
The leading-order term of the matrix in the argument of this determinant as $z_i\to 0$ is 
$$
    (\zeta_{i,j}' \mbox{I}_{\E^{i,j}} + {z_{i,j}'} (m_i-1) \tilde{A}_{i,j}^{m_i-1} ) \lambda_i; 
$$
in particular, if for any $j'$ such that $\zeta_{i,j'}^{m_i} = \zeta_{i,j}^{m_i}$ we let 
\begin{equation}\label{eq:point-on-exceptional-divisor}
  (z_{i,j}':\zeta_{i,j}') = ( 1 : \zeta_{i,j'}^{m_i-1} )
\end{equation}
then the above matrix is singular. 
Let us introduce a further subscript to $\zeta_{i,j'}^{m_i-1}$ in order to remember the information that 
$\zeta_{i,j'}^{m_i} = \zeta_{i,j}^{m_i}$: set 
$$
  \zeta_{i,j,j'}^{m_i-1} = \zeta_{i,j'}^{m_i-1}. 
$$
Said differently, the proper transform of $\Sigma$ by $\sigma_{i,j}$ intersects $E_{i,j}$ precisely in the points 
$$
  ( 1 : \zeta_{i,j,j'}^{m_i-1} )
$$
Therefore, we may recursively define a sequence of blow-ups of $Z$ by the following 
\begin{construct}\label{constr:blowup}
\begin{enumerate}
\item we define 
$$
  \sigma_1 :Z^1 \longrightarrow Z 
$$
as the fibre product of the monodial transformations $\sigma_{i,j}$ for all $1\leq i \leq n$ with $m_i\geq 1$
and eigenvalues $\zeta_{i,j}^{m_i}$ of $-m_i A_{i,j}^{m_i}$; 
\item next we blow up $Z^1$ at the points on $E_{i,j}$ corresponding to the points 
(\ref{eq:point-on-exceptional-divisor}) of intersection with the proper transform of $\Sigma$ 
along $\sigma_1$ (these are therefore indexed by joint eigenvalues of $A_i^{m_i}, A_i^{m_i-1}$) 
for all $i$ such that $m_i \geq 2$, denote the blow-up by 
$$
  \sigma_{2} : Z^{2} \longrightarrow Z^{1}; 
$$
\item in the $(m+1)$'th step for all $i$ with $m_i\geq m$ and all eigenvalue 
$$
  \zeta_{i,j,j',\ldots, j^{(m)}}^{m_i-m} = \zeta_{i,j^{(m)}}^{m_i-m}
$$
of $(m-m_i)A_i^{m_i-m}$ we blow up $Z^{m}$ at the intersection point 
\begin{equation}\label{eq:point-on-exceptional-divisor-m}
  (z_{i,j,j'\ldots , j^{(m-1)}}':\zeta_{i,j,j'\ldots , j^{(m-1)}}') 
    = ( 1: \zeta_{i,j,j',\ldots, j^{(m)}}^{m_i-m} ) \in E_{i,j,j'\ldots , j^{(m-1)}}
\end{equation}
of the proper transform of $\Sigma$ along $\sigma_{m-1}$ with the exceptional divisors 
of $\sigma_{m-1}$ projecting to the point $p_i$, call the resulting exceptional divisor 
$$
  E_{i,j,j'\ldots , j^{(m)}}
$$
with homogeneous coordinates 
$$
  (z_{i,j,j'\ldots , j^{(m)}}':\zeta_{i,j,j'\ldots , j^{(m)}}'), 
$$
and set 
$$
  U_{i,j,j'\ldots , j^{(m)}}' = (z_{i,j,j'\ldots , j^{(m)}}' \neq 0) \subset Z^{m};
$$
\item finally we define (\ref{eq:master-surface}) as $Z^M$ for the value $M = \max_{1\leq i \leq n} (m_i)$ 
with 
$$
  \tilde{\sigma} = \sigma_1 \circ \cdots \circ \sigma_M. 
$$
\end{enumerate}
\end{construct}
Notice that the degree of $\tilde{Q}_{i,j}$ appearing in (\ref{eq:theta-tilde-local}) is one less than that of $Q_i$; 
similarly, in the $(m+1)$'th step of the recursive procedure the proper transform of $\Sigma$ in a local chart 
is defined by a section 
\begin{equation}\label{eq:proper-transform-spectral-curve-m}
 \det\left( \zeta_{i,j,j',\ldots , j^{(m)}}'  \lambda_i \mbox{I}_{\E^{i,j,j',\ldots , j^{(m)}}} -  
	z_{i,j,j',\ldots , j^{(m)}}'z_i^{-m}\tilde{\theta}^{i,j,j',\ldots, j^{(m)}} \right) 
\end{equation}
for some 
\begin{align}
  \tilde{\theta}^{i,j,j',\ldots, j^{(m)}} & = \left( \d \tilde{Q}_{i,j,j',\ldots ,  j^{(m)}} + \Lambda_{i,j,j',\ldots ,  j^{(m)}} (z_i)^{-1} 
	    + \mbox{ holomorphic terms} \right) \d z_i \label{eq:transformed-irregular-part}\\
   \tilde{Q}_{i,j,j',\ldots, j^{(m)}} & = {A}_{i,j,j',\ldots, j^{(m)}}^{m_i-m-1} z_i^{-m_i+m+1} + \cdots + {A}_{i,j, j',\ldots, j^{(m)}}^1 z_i^{-1} 
    \notag 
\end{align}
where the matrices on the right hand side of the latter equation are restrictions of $A_i^{m_i-m-1},\ldots, A_i^1,\Lambda_i$ to 
joint eigenspaces of 
$$
  -m_i {A}_{i}^{m_i},\ldots , (m-m_i) {A}_{i}^{m_i-m}
$$
for the eigenvalues 
$$
  \zeta_{i,j}^{m_i}, \zeta_{i,j,j'}^{m_i-1}, \ldots ,  \zeta_{i,j,j',\ldots, j^{(m)}}^{m_i-m}
$$
respectively. 
We will denote the dimension of this joint eigenspace by 
\begin{equation}\label{eq:simultaneous-eigenspace-dimension}
  d_{i,j,j',\ldots ,j^{(m)}}. 
\end{equation}
Now just as in (\ref{eq:second-order-parabolic-map}), the map $\tilde{\theta}^{i,j,j',\ldots, j^{(m_i-1)}}$ 
is of the form 
\begin{equation}\label{eq:higher-order-parabolic-map}
  z_i^{-m_i} \tilde{\theta}^{i,j,j',\ldots, j^{(m_i-1)}} = (\Lambda_{i,j,j',\ldots ,  j^{(m_i-1)}} + O(z_i)) \lambda_i, 
\end{equation}
hence it gives rise to maps 
$$
  z_i^{-m_i} \tilde{\theta}^{i,j,j',\ldots, j^{(m_i-1)}}_{\alpha} : (p\circ \tilde{\sigma})^* \E^{i,j,j',\ldots , j^{(m_i-1)}}_{\alpha}
  \otimes L^{\vee} \longrightarrow (p\circ \tilde{\sigma})^* \E^{i,j,j',\ldots , j^{(m_i-1)}}_{\alpha}
$$
for every $\alpha\in\R$ and $i,j,j',\ldots, j^{(m_i-1)}$, where $\E^{i,j,j',\ldots , j^{(m)}}_{\alpha}$ 
stands for the filtration on the corresponding simultaneous eigenspace obtained from the parabolic structure 
by virtue of Claim \ref{clm:A-preserves-F}. 
Now just as we refined the formula (\ref{eq:proper-transformed-spectral-curve}) defining the 
proper transform of the spectral curve to define a parabolic sheaf (\ref{eq:parabolic-spectral-sheaf}) 
we may again refine (\ref{eq:proper-transform-spectral-curve-m}) to define coherent sheaves on $\widetilde{Z}$ by 
the formula (\ref{eq:spectral-sheaf-E-alpha}) away from the charts $U_{i,j,j'\ldots , j^{(m_i-1)}}'$ and by 
\begin{equation}\label{eq:higher-order-parabolic-spectral-sheaf}
  S_{\E_{\alpha}} = \coker  \left( \zeta_{i,j,j',\ldots , j^{(m_i-1)}}'  \lambda_i \mbox{I}_{\E^{i,j,j',\ldots , j^{(m_i-1)}}_{\alpha}} -  
	z_{i,j,j',\ldots , j^{(m_i-1)}}'z_i^{-m_i}\tilde{\theta}^{i,j,j',\ldots, j^{(m_i-1)}}_{\alpha} \right) 
\end{equation}
on $U_{i,j,j'\ldots , j^{(m_i-1)}}'$, the argument being a section of the bundle 
\begin{align*}
  (p\circ \tilde{\sigma})^* & \left( \End (\E^{i,j,j',\ldots , j^{(m_i-1)}}_{\alpha}) \otimes L \right) \otimes \tilde{\sigma}^* \O_Z(1) \\ 
   & \otimes \O_{\widetilde{Z}} (- E_{i,j,j',\ldots, j^{(m_i-1)}} ). 
\end{align*}

The results of this section can be summarized as follows. 

\begin{thm}\label{thm:parabolicBNR}
 The correspondence 
 $$
  (\E_{\bullet},\theta ) \mapsto S_{\E_{\bullet}}
 $$
 establishes an equivalence of categories between 
 \begin{enumerate}
  \item on the one hand, parabolic Higgs bundles of rank $r$ on $\CP$ with \label{prop:parabolicBNR1}
    \begin{enumerate}
     \item for all $1\leq i \leq n$ an irregular singularity with pole of order $m_i$ at $p_i$ locally of the form (\ref{eq:theta-local}) with (\ref{eq:Qi}) \label{prop:parabolicBNR12}
	with semi-simple matrices $A_i^m$ and endowed with a compatible parabolic structure,
      \item and a reduced irreducible spectral curve \label{prop:parabolicBNR13}
    \end{enumerate}
  \item and on the other hand, $\R$-parabolic pure sheaves $S_{\alpha}$ of dimension $1$ and rank $1$ with parabolic divisor 
      $$
	(p \circ \tilde{\sigma})^{-1}(D_{\red})
      $$
      on $\widetilde{Z}$, with support $\widetilde{\Sigma}$ satisfying the following properties: \label{prop:parabolicBNR2}
  \begin{enumerate}
     \item $\widetilde{\Sigma}$ is reduced and irreducible \label{prop:parabolicBNR21}
     \item $\widetilde{\Sigma}$ is generically $r$ to $1$ over $C$, \label{prop:parabolicBNR22}
     \item $\widetilde{\Sigma} \cap (\xi = 0) = \emptyset$, \label{prop:parabolicBNR23}
     \item for all $1\leq i \leq n$ we have $\widetilde{\Sigma}\cap E_{i,j,j',\ldots, j^{(m_i-1)}}\subset U_{i,j,j',\ldots, j^{(m_i-1)}}'$, 
      i.e. the support does not pass through the point at infinity of the exceptional divisor $E_{i,j,j',\ldots, j^{(m_i-1)}}$,
      \label{prop:parabolicBNR24}
      \item for all $1\leq i \leq n$ the intersection $\widetilde{\Sigma}\cap E_{i,j,j',\ldots, j^{(m_i-1)}}$ 
	consists of $d_{i,j,j',\ldots, j^{(m_i-1)}}$ points counted with multiplicity, \label{prop:parabolicBNR24,5}
      \item for all $1\leq i \leq n$ and all $m<m_i$ the proper transform of $E_{i,j,j',\ldots, j^{(m-1)}}$ in $\widetilde{Z}$ 
      does not intersect $\widetilde{\Sigma}$. 
	\label{prop:parabolicBNR25}
   \end{enumerate}
 \end{enumerate}
\end{thm}

\begin{proof}
 The functor (\ref{prop:parabolicBNR1}) $\rightarrow$ (\ref{prop:parabolicBNR2}) is given by (\ref{eq:higher-order-parabolic-spectral-sheaf}). 
 The sheaves $S_{\E_{\bullet}}$ satisfy properties (\ref{prop:parabolicBNR21}) by assumption and (\ref{prop:parabolicBNR22}) 
 because $\E$ is of rank $r$. 
 Furthermore as $\widetilde{\Sigma}$ is the spectral curve of a regular map 
 $$
  \E \to \E \otimes K_C(D), 
 $$
 it stays in the affine subset $\xi\neq 0$ of $Z$. 
 By (\ref{eq:higher-order-parabolic-map}) $\widetilde{\Sigma}$ intersects the exceptional divisor 
 $E_{i,j,j',\ldots, j^{(m_i-1)}}$ in the points 
 \begin{equation}\label{eq:eigenvalue-residue}
  (z_{i,j,j'\ldots , j^{(m_i-1)}}':\zeta_{i,j,j'\ldots , j^{(m_i-1)}}') 
    = ( 1: \lambda_{i,j,j',\ldots, j^{(m_i-1)}, j^{m_i}} )
 \end{equation}
 where the $\lambda_{\cdots}$ on the right hand side stand for the eigenvalues of $\Lambda_i$ restricted to the 
 joint eigenspaces of $A_i^{m_i},\ldots, A_i^1$, and this shows (\ref{prop:parabolicBNR24}) and (\ref{prop:parabolicBNR24,5}). 
 Next, (\ref{prop:parabolicBNR25}) holds because to define $\widetilde{Z}$ in Construction \ref{constr:blowup} 
 we blew up all points (\ref{eq:point-on-exceptional-divisor-m}) of intersection of the proper transform of $\Sigma$ 
 with $E_{i,j,j',\ldots, j^{(m-1)}}$ and the support of the sheaves $S_{\E_{\bullet}}$ is the proper transform 
 $\widetilde{\Sigma}$. 
 Purity of $S_{\E_{\bullet}}$ follows the exact same argument as in the case $m_i=1$ (c.f. Claim \ref{clm:R-parabolic}) 
 based on the Auslander--Buchsbaum formula. 
 Finally, again as in the case $m_i=1$ the identity 
 $$
  (p \circ \tilde{\sigma})^* \O_C(D_{\red}) = \O_{\widetilde{Z}} ((p \circ \tilde{\sigma})^{-1}(D_{\red}))
 $$
 immediately shows that the sheaves $S_{\E_{\bullet}}$ are $\R$-parabolic with divisor $(p \circ \tilde{\sigma})^{-1}(D_{\red})$. 
 
 The quasi-inverse is the direct image functor 
 \begin{equation}\label{eq:direct-image-of-spectral-sheaf}
  \E_{\alpha} = (p \circ \sigma)_* S_{\E_{\alpha}}, \quad 
    \theta_{\alpha} = \pi_* (\zeta \cdot )
 \end{equation}
 by a straightforward generalisation of Claim \ref{clm:R-parabolic} to the case of order $m_i\geq 2$. 
 There remains to check that 
 \begin{itemize}
 \item the irregular part of $\theta$ is of the form (\ref{eq:theta-local}), (\ref{eq:Qi}) and 
 \item that the residue $\Lambda_i$ of $\theta$ at $z_i=0$ respects the parabolic filtration of 
 $\E_{\bullet}$ at $p_i$. 
 \end{itemize}
 For the proof of these statements observe that by the construction of $\E_{\alpha}$ it has a local splitting 
  $$
   \E_{\alpha} = \oplus_{j, j', \ldots, j^{(m_i-1)}} \E^{i, j, j', \ldots, j^{(m_i-1)}}_{\alpha}
 $$
 according to the support of sections of $S_{\E_{\alpha}}$: local sections of $\E^{i, j, j', \ldots, j^{(m_i-1)}}_{\alpha}$ 
 are defined as the push-forward of local sections of $S_{\E_{\alpha}}$ supported on the branches of 
 $\widetilde{\Sigma}$ intersecting a given exceptional divisor $E_{i,j,j',\ldots, j^{(m_i-1)}}$. 
 From (\ref{eq:blow-up-coordinates-2}), on $U_{i,j}'$ we deduce the identity 
 \begin{equation}\label{eq:zeta_i}
   \zeta_i = \zeta_{i,j}^{m_i} + z_i \frac{\zeta_{i,j}'}{z_{i,j}'}. 
 \end{equation}
 Here by the definition of $U_{i,j}'$ we may normalize $z_{i,j}'=1$. 
 In the case $m_i > 1$ Construction \ref{constr:blowup} proceeds by applying a blow-up 
 $\sigma_{i,j,j'}$ at the point $\zeta_{i,j}' = \zeta_{i,j,j'}^{m_i-1}$ of $E_{i,j}$ given by the formula 
 $$
  z_i \zeta_{i,j,j'}' = z_{i,j,j'}' (\zeta_{i,j}' - \zeta_{i,j,j'}^{m_i-1}).
 $$
 We may again set $z_{i,j,j'}' = 1$ on $U_{i,j,j'}'$ and solve this expression for $\zeta_{i,j}'$: 
 $$
  \zeta_{i,j}' = \zeta_{i,j,j'}^{m_i-1} + z_i \zeta_{i,j,j'}'. 
 $$
 Plugging this into (\ref{eq:zeta_i}) yields 
 $$
  \zeta_i = \zeta_{i,j}^{m_i} + z_i \zeta_{i,j,j'}^{m_i-1} +  z_i^2 \zeta_{i,j,j'}'. 
 $$
 By induction on $m_i$, this argument shows on $U_{i,j,j'\ldots , j^{(m_i-1)}}'$ 
 (using the normalization $z_{i,j,j'\ldots , j^{(m_i-1)}}'=1$) the identity 
 $$
  \zeta_i = \zeta_{i,j}^{m_i} + z_i \zeta_{i,j,j'}^{m_i-1} + \cdots + z_i^{m_i-1} \zeta_{i,j,j'\ldots , j^{(m_i-1)}}^1 
    + z_i^{m_i} {\zeta_{i,j,j'\ldots , j^{(m_i-1)}}'}
 $$
 Here $\zeta_{i,j}^{m_i}, \ldots, \zeta_{i,j,j'\ldots , j^{(m_i-1)}}^1$ are constants, therefore 
 the restriction to $\E^{i, j, j', \ldots, j^{(m_i-1)}}_{\alpha}$ of the direct image of the multiplication map by 
 $\zeta = \zeta_i \lambda_i$ has the expansion 
 $$
  \left( \frac{\zeta_{i,j}^{m_i}}{z_i^{m_i+1}} + \cdots + \frac{\zeta_{i,j}^{1}}{z_i^2}\right) \Id_{\E^{i, j, j', \ldots, j^{(m_i-1)}}_{\alpha}} \d z 
  + O \left( z_i^{-2+\frac 1r}\right) \d z, 
 $$
 which is precisely the first assertion. Furthermore, the residue of $\theta$ at $p_i$ is obtained as the direct image of the 
 multiplication map by 
 $$
  \zeta_{i,j,j'\ldots , j^{(m_i-1)}}' 
 $$
 which is a local coordinate of the surface $\widetilde{Z}$. As $S_{\E_{\alpha}}$ is a sheaf of 
 $\O_{\widetilde{Z}}$-modules, it is clearly preserved by multiplication by $\zeta_{i,j,j'\ldots , j^{(m_i-1)}}'$.  
 Therefore $\res_{p_i}(\theta)$ preserves $\E_{\alpha}$, which is the second assertion. 
 
\end{proof}

\section{Poisson isomorphism}\label{sec:Poisson}

In this section we prove that the natural holomorphic Poisson structures on the moduli spaces associated 
to the groupoids appearing in Theorem \ref{thm:parabolicBNR} are preserved by the correspondence of the Theorem. 
The proof closely follows that of Proposition 5.1 \cite{Sz-plancherel}. 
We start by defining these Poisson structures. 

\subsection{Irregular Dolbeault moduli space}
Let us first treat the moduli space of stable irregular parabolic Higgs bundles on $C$ with fixed semi-simple 
irregular part  (\ref{eq:theta-local}), (\ref{eq:Qi}) and residue in a fixed regular orbit: 
by \cite{Biq-Boa} it carries an Atiyah--Bott hyper-K\"ahler structure, in particular for the Dolbeault holomorphic 
structure $I$ it admits a holomorphic symplectic structure. 
These irregular Dolbeault moduli spaces $M_{\Dol}^{\irr}$ may then be put into a family 
\begin{equation}\label{eq:Dolbeault-moduli}
  \Mod_{\Dol}^{\irr} \to \C^{r\sum_i (m_i+1)}
\end{equation}
by varying the eigenvalues of the matrices $A_i^m$ and of $\Lambda_i$ arbitrarily. 
In concrete terms, $\Mod_{\Dol}^{\irr}$ represents the functor from Artinian schemes over $\C$ to sets mapping 
\begin{itemize}
 \item a scheme $T$ to the set of isomorphism classes of parabolically stable pairs $(\bar{\E}, \bar{\theta})$ where $\bar{\E}$ is 
a regular vector bundle over $T \times C$ and 
$$
  \bar{\theta} \in H^0(T \times C, \End(\bar{\E}) \otimes_{\O_{T \times C}} p_C^* L), 
$$
with $p_C:T \times C \to C$ the projection map;
\item  and a morphism $S\to T$ of such schemes to the set of maps of Higgs bundles parameterized by $S$ and $T$ respectively. 
\end{itemize}
The condition of parabolic stability is a parabolic version of slope-stability: one first introduces the notion of 
parabolic degree as the sum of the usual degree and all the parabolic weights, and then induces parabolic 
weights and parabolic degree on sub-objects; for further details see \cite{Boa-survey}. 
The map in (\ref{eq:Dolbeault-moduli}) associates to $(\E,\theta)$ the eigenvalues of the polar part of $\theta$ 
considered as a meromorphic section of $\End(\E)\otimes_{\O_C}K_C$. 
\begin{rk}\label{rk:dR-Dol-moduli}
\begin{enumerate}
 \item To the knowledge of the author, this irregular Dolbeault moduli functor has not yet been studied algebraically. 
  Notice nonetheless that E. Markman studied in \cite{Mark} a related moduli problem, which is essentially the same as our functor above except that 
  there the parabolic structure is not present and the irregular type may be twisted (non-semi-simple). 
\item An analogous de Rham moduli space is constructed algebraically in \cite{inaba-saito}; 
  however, as it is explained in Remark 1.2 op. cit., the data they fix at the singularities is somewhat different from the usual 
  irregular type, and this latter cannot be determined from it. On the other hand, Inaba and Saito consider more general 
  families of puctured curves 
  $$
    \mathcal{C} \to T
  $$
  instead of just a product. 
\item Depending on the eigenvalues of the singular parts, the moduli space $M_{\Dol}^{\irr}$ might be empty. 
For instance, as it readily follows from the residue theorem, this will be the case unless the sum of the traces of all the residues is an integer $-\delta$. 
In this case, we have $\deg(\E) = \delta$. 
However, the above consequence of the residue theorem is just a first easy condition in deciding whether or not the moduli space is empty --- 
in general, this problem, called irregular Deligne--Simpson problem is hard. 
The correspondence of Theorem \ref{thm:parabolicBNR} holds (but is vacuous) even if the irregular moduli space is empty. 
\end{enumerate}
\end{rk}
Elements of $\C^{r\sum_i (m_i+1)}$ will be denoted as nested sequeces of eigenvalues 
$$
  (\zeta_{i,j}^{m_i},\ldots ,\zeta_{i,j,j'\ldots , j^{(m_i-1)}}^1, \lambda_{i,j,j'\ldots , j^{(m_i)}})_{i,j,j'\ldots , j^{(m_i)}}
$$
where the $\zeta_{i,j,j'\ldots , j^{(m)}}^m$ denote eigenvalues of $-m A^m_i$ for $0\leq m < m_i$ and 
$\lambda_{i,j,j'\ldots , j^{(m_i)}}$ is an eigenvalue of $\Lambda_{i,j,j',\ldots ,  j^{(m_i-1)}}$ (\ref{eq:eigenvalue-residue}). 
One could equally let the parabolic weights $\alpha_i^l$ vary and this would give further deformation parameters, 
but we will ignore this point here. 
The smooth part of $\Mod_{\Dol}^{\irr}$ is a holomorphic Poisson manifold with Poisson bivector field denoted by $\Pi_{\Dol}$. 
Let us be more specific concerning this holomorphic Poisson structure, following \cite{bis-ram} where the non-singular case 
is treated. 
It is easy to see that the deformation theory of an irregular Higgs bundle $(\E,\theta)$ in $\Mod_{\Dol}^{\irr}$ 
is governed by the hypercohomology spaces $\H^{\bullet}$ of the complex 
\begin{equation}\label{eq:deformation-complex}
  \End(\E) \xrightarrow{\ad_{\theta}} \End(\E) \otimes L 
\end{equation}
with $L=K_C(D)$. 
Infinitesimal deformations are given by the first hypercohomology space $\H^{1}$ and $\H^2$ is the obstruction space. 
By Grothendieck duality, the cotangent space of (\ref{eq:Dolbeault-moduli}) is then given by the first hypercohomology of the complex 
$$
   \End(\E) \otimes \O_C(-D) \xrightarrow{\ad_{\theta}} \End(\E) \otimes K_C. 
$$
Given two cotangent vectors 
$$
  [T],[X]\in\H^{1} (\End(\E) \otimes \O_C(-D) \to \End(\E) \otimes K_C) 
$$
represented by endomorphism-valued $1$-forms $T,X$ their cup product 
$$
  [T]\cup [X] = [T \wedge X]
$$
belongs to the second hypercohomology group of the complex 
\begin{equation}\label{eq:C-complex}
  C^0 \xrightarrow{(\ad_{\theta}\otimes \Id,\Id\otimes \ad_{\theta})} C^1 \xrightarrow{\Id \otimes \ad_{\theta} - \ad_{\theta} \otimes \Id} C^2
\end{equation}
with 
\begin{align*}
  C^0  = & (\End(\E) \otimes \O_C(-D)) \otimes (\End(\E) \otimes \O_C(-D)) \\ 
  C^1 = & [(\End(\E) \otimes K_C) \otimes (\End(\E) \otimes \O_C(-D))] \\ 
  & \oplus [(\End(\E)\otimes \O_C(-D)) \otimes (\End(\E) \otimes K_C)] \\ 
  C^2 = & (\End(\E) \otimes K_C) \otimes (\End(\E) \otimes K_C). 
\end{align*}
The $\ad$-invariant symmetric bilinear form on $\Gl_r$ 
\begin{equation}\label{eq:Killing-form}
  B: \varphi, \psi \mapsto \tr(\varphi \psi) 
\end{equation}
induces a chain map $\Phi$ from (\ref{eq:C-complex}) to 
\begin{equation}\label{eq:L-complex}
  0 \to K_C(-D) \to 0. 
\end{equation}
In concrete terms, to a cocycle 
$$
  (\varphi_1 \otimes \psi_1 ,  \psi_2  \otimes \varphi_2) \in C^1
$$
the map $\Phi$ associates the section 
$$
  B (\varphi_1, \psi_1) + B (\psi_2, \varphi_2). 
$$
Plainly $\Phi$ is a chain map because the bilinear form $B$ is $\ad$-invariant 
$$
  B(\ad_{\theta}(\eta) , \zeta ) + B(\eta , \ad_{\theta}(\zeta)) = 0. 
$$
It then follows that the image of $[T]\cup [X]$ by $\Phi$ defines a degree 
$2$ hypercohomology class in (\ref{eq:L-complex}), i.e. a class in 
\begin{equation}\label{eq:H1K}
   H^1(C,K_C(-D)) \cong (H^0(C, \O(D)))^{\vee}. 
\end{equation}
The dual of the vector space on the right hand side fits into a short exact sequence 
\begin{equation}\label{eq:H0SES}
  0 \to H^0(C, \O) \to H^0(C, \O(D)) \to \oplus_{i=1}^n \C_{p_i}^{m_i+1} \to 0, 
\end{equation}
where $\C_p$ stands for the skyscraper sheaf supported at the point $p\in C$. 
In particular, $H^0(C, \O(D))$ contains the element $1\in H^0(C, \O)$, hence any element in 
(\ref{eq:H1K}) can be evaluated on this class. We may therefore define an alternating bilinear map by 
\begin{align*}
  \Pi_{\Dol}: T^*\Mod_{\Dol}^{\irr} \times T^*\Mod_{\Dol}^{\irr} & \to \C \\
  ([T],[X]) & \mapsto \langle \Phi ([T] \cup [X]) , 1 \rangle  
\end{align*}
where $\langle .,. \rangle$ stands for Serre duality (\ref{eq:H1K}). 
As usual, this formula is the reduction of an infinite-dimensional flat pairing on an $L^2$-space 
of $1$-forms with values in the endomorphisms of the smooth vector bundle underlying $\E$, 
so the Schouten-bracket $[\Pi_{\Dol},\Pi_{\Dol}]$ of $\Pi_{\Dol}$ with itself is $0$. 
Since using the Dolbeault resolution of $K_C$ Serre duality is defined by integration of $2$-forms on 
$C$ and $\Phi$ is given by (\ref{eq:Killing-form}) we infer that the restriction of this pairing to the 
symplectic leaves reads 
$$
  \int_C \tr(T\wedge X)
$$
which is the usual Atiyah--Bott holomorphic symplectic form. 
Therefore, $\Pi_{\Dol}$ defines the holomorphic Poisson structure we were looking for. 
Let us point out that the vector spaces in (\ref{eq:H0SES}) localised at the points $p_i$ 
correspond to the infinitesimal modifications of some Casimir operators, i.e. to tangent vectors of  
the parameter space $\C^{r\sum_i (m_i+1)}$. 
\begin{rk}
 The Poisson structure $\Pi_{\Dol}$ essentially matches up with that of the Main Theorem of \cite{Mark} corresponding to the section 
 $$
  1 \in H^0(C,\O_C(D)) \cong H^0(C,L\otimes K^{-1}).
 $$
\end{rk}

\subsection{Relative Picard bundles}
We now turn our attention to the category of sheaves on $\widetilde{Z}$ satisfying the properties 
(\ref{prop:parabolicBNR21})--(\ref{prop:parabolicBNR25}) of Theorem \ref{thm:parabolicBNR}. 
Observe first that $Z$ is a Poisson surface for the canonical Liouville $2$-form on the total space 
of the canonical line bundle $K_C$. The degeneracy divisor of this Poisson structure on $Z$ is given by 
$$
  \pi^{-1} D + 2 (\xi )
$$
Therefore, the pull-back of this $2$-form to $\widetilde{Z}$ by $\tilde{\sigma}$ also defines a Poisson structure. 
Let us determine its degeneracy divisor: by differentiating (\ref{eq:blow-up-coordinates-2}) twice 
(and as usual setting $z_{i,j}'=1$ on $U_{i,j}'$) one easily derives the formula 
$$
  z_i \d z_i \wedge \d \zeta_{i,j}' = \d z_i \wedge \d \zeta_i. 
$$
As we have already noticed after (\ref{eq:Uil'}), on $U_{i,j}'$ the equation $z_i = 0$ defines precisely the exceptional divisor $E_{i,j}$. 
By an abuse of notation let $F|_{p_i}$ and $E_{i,j,j',\ldots, j^{(m-1)}}$ denote the proper transforms of the fibre 
$F|_{p_i}$ and of the exceptional divisor $E_{i,j,j',\ldots, j^{(m-1)}}$ with respect to the various iterated blow-ups 
of Construction \ref{constr:blowup}. 
(For the exceptional divisors this is obviously only applicable for $\sigma_{m'}$ with $m'>m$.)
We infer that the pull-back by $\sigma_{i,j}$ of the canonical $2$-form has a pole of order one less on $E_{i,j}$ than on $F|_{p_i}$. 
Now by an easy induction argument we can show that the pull-back by $\sigma_1 \circ \sigma_2$ of the 
canonical $2$-form has a pole of order $2$ less on $E_{i,j,j'}$ than on $F|_{p_i}$, and so on, the pull-back by 
$\sigma_1 \circ \cdots \circ \sigma_{m_i}$ of the canonical $2$-form has a pole of order $m_i$ less on 
$E_{i,j,j',\ldots, j^{(m_i-1)}}$ than on $F|_{p_i}$. 
It follows that the degeneracy divisor of the pull-back of the canonical Poisson structure of $Z$ to $\widetilde{Z}$ is given by 
\begin{align}
  D_{\infty} & = \pi^{-1} D + 2 (\xi ) - \left( \sum_{i,j} E_{i,j} + 
    (\sum_{j'} 2 E_{i,j,j'} + ( \cdots + \sum_{j^{(m_i-1)}} m_i E_{i,j,j',\ldots, j^{(m_i-1)}} ) \cdots ) \right), \notag \\
  & =  2 (\xi ) + \sum_i (m_i + 1) F|_{p_i} +  (\sum_{j} m_i E_{i,j} + (\sum_{j'} (m_i - 1) E_{i,j,j'} + 
  ( \cdots + \sum_{j^{(m_i-1)}} E_{i,j,j',\ldots, j^{(m_i-1)}} ) \cdots )) \label{eq:degeneracy-divisor}
\end{align}
Consider first the Hilbert scheme 
\begin{equation}\label{eq:Hilbert-scheme}
   \Hilb(\widetilde{Z},H)
\end{equation}
of curves on $\widetilde{Z}$ with a given Hilbert polynomial $H$. 
Specifically, the Picard group of $\widetilde{Z}$ is generated by the fibre class $F$, the class of the infinity section 
$C_{\infty}$ and the classes of the exceptional divisors $E_{i,j,j',\ldots, j^{(m-1)}}$ of the blowups 
$\sigma_1,\ldots ,\sigma_M$. An ample line bundle on $\widetilde{Z}$ is given by 
$$
  \mathcal{L} = \O_{\widetilde{Z}}\left( -\sum_{m,i,j,j',\ldots, j^{(m-1)}} E_{i,j,j',\ldots, j^{(m-1)}} \right) \otimes 
	\tilde{\sigma}^*\O_{Z}(1) \otimes (p \circ \tilde{\sigma})^*\O_{C}(1). 
$$
The intersection form on the second homology of $\widetilde{Z}$ is non-degenerate and we may consider the homology 
class dual to the class of the divisor 
$$
  r F + \sum_{i,j,j',\ldots ,j^{(m_i-1)}} d_{i,j,j',\ldots ,j^{(m_i-1)}} E_{i,j,j',\ldots ,j^{(m_i-1)}}
$$
where we recall that $d_{i,j,j',\ldots ,j^{(m_i-1)}}$ was defined in (\ref{eq:simultaneous-eigenspace-dimension}) 
as the dimension of the simultaneous eigenspace of the matrices $-m_iA_i^{m_i},\ldots , -A_i^1$ for the eigenvalues 
$$
    \zeta_{i,j}^{m_i}, \zeta_{i,j,j'}^{m_i-1}, \ldots ,  \zeta_{i,j}^{1}
$$
respectively. 
The generic curve in this class will then intersect the generic fiber of $\widetilde{Z}$ in $r$ points, 
the exceptional divisor $E_{i,j,j',\ldots ,j^{(m_i-1)}}$ in $d_{i,j,j',\ldots ,j^{(m_i-1)}}$ points counted with 
multiplicity, and will be disjoint from $C_{\infty}$ and the exceptional divisors $E_{i,j,j',\ldots ,j^{(m-1)}}$ 
with $m<m_i$. 
In different terms, such a curve satisfies the conditions (\ref{prop:parabolicBNR21})--(\ref{prop:parabolicBNR25}) 
of Theorem \ref{thm:parabolicBNR}. 
We then pick $H$ to be the Hilbert polynomial with respect to $\mathcal{L}$ of a curve in this class.
Notice that the generic curve in this family intersects $E_{i,j,j',\ldots ,j^{(m_i-1)}}$ in 
$d_{i,j,j',\ldots ,j^{(m_i-1)}}$ distinct points, and is smooth. 
Let us denote by $B$ the subscheme of (\ref{eq:Hilbert-scheme}) parameterizing smooth curves. 
Notice however that the curves having nodal singularities at some points of $E_{i,j,j',\ldots ,j^{(m_i-1)}}$ are 
also of interest, for they correspond via Theorem \ref{thm:parabolicBNR} to non-regular residues 
$\Lambda_{i,j,j',\ldots ,j^{(m_i-1)}}$ of the Higgs field. 

Next, let us consider the compactified Picard variety \cite{Alt-Kle} corresponding to the family (\ref{eq:Hilbert-scheme}) 
\begin{equation}\label{eq:Picard}
   \Pic_{\rel}(\widetilde{Z},H,d) \to \Hilb(\widetilde{Z},H)
\end{equation}
parameterizing torsion-free coherent sheaves of rank $1$ and degree $d$ on the fibers of (\ref{eq:Hilbert-scheme}). 
It follows from \cite{Don-Mar} that the restriction 
$$
   \Pic_{\rel}^0(\widetilde{Z},H,d) \to B
$$
of (\ref{eq:Picard}) to the space $B$ parameterizing smooth connected curves carries a canonical holomorphic Poisson structure 
with Poisson bivector field denoted by $\Pi_{\Pic^0}$. 
Let us describe explicitly $\Pi_{\Pic^0}$: for this purpose, notice that the Zariski tangent space of 
$\Pic_{\rel}^0(\widetilde{Z},d,H)$ at a given sheaf $S$ is given by $\GlobExt_{\O_{\widetilde{Z}}}^1(S,S)$. 
By Grothendieck-duality, the Zariski cotangent space is then isomorphic to 
$$
  \GlobExt_{\O_{\widetilde{Z}}}^1(S,K_{\widetilde{Z}}\otimes S) \cong  \GlobExt_{\O_{\widetilde{Z}}}^1(K_{\widetilde{Z}}^{\vee}\otimes S,S).  
$$
Thus the Yoneda product 
$$
  \GlobExt_{\O_{\widetilde{Z}}}^1(K_{\widetilde{Z}}^{\vee}\otimes S,S) \times \GlobExt_{\O_{\widetilde{Z}}}^1(S,K_{\widetilde{Z}}\otimes S) 
  \to \GlobExt_{\O_{\widetilde{Z}}}^2(K_{\widetilde{Z}}^{\vee}\otimes S,K_{\widetilde{Z}}\otimes S) 
$$
induces by duality an alternating map 
\begin{equation}\label{eq:dual-Yoneda}
  \GlobExt_{\O_{\widetilde{Z}}}^2(K_{\widetilde{Z}}^{\vee}\otimes S,K_{\widetilde{Z}}\otimes S)^{\vee} \to 
  \GlobExt_{\O_{\widetilde{Z}}}^1(S,S) \times \GlobExt_{\O_{\widetilde{Z}}}^1(S,S). 
\end{equation}
Denote the support of $S$ by $\widetilde{\Sigma}$; by assumption, this is a smooth connected curve. 
\begin{clm}\label{clm:Poisson}
 The space 
 $$
  \GlobExt_{\O_{\widetilde{Z}}}^2(K_{\widetilde{Z}}^{\vee}\otimes S,K_{\widetilde{Z}}\otimes S)^{\vee} 
 $$
 admits a subspace isomorphic to $H^0(\widetilde{\Sigma}, \C)$.
\end{clm}

\begin{proof}
 This is a standard application of the spectral sequence abutting to $\GlobExt_{\O_{\widetilde{Z}}}^2$ and of 
 Serre duality on $\widetilde{\Sigma}$, making use of the identification 
 $$
  N_{\widetilde{\Sigma}|\widetilde{Z}} \cong K_{\widetilde{\Sigma}}(-\widetilde{\Sigma} \cap D_{\infty}) 
 $$
 provided by the Poisson structure of $\widetilde{Z}$. 
 For details see Lemma 5.3 \cite{Sz-plancherel}. 
\end{proof}

The Poisson bivector field $\Pi_{\Pic^0}$ on (\ref{eq:Picard}) is then defined as the image under 
(\ref{eq:dual-Yoneda}) of the canonical generator 
$$
  1\in H^0(\widetilde{\Sigma}, \C) \subset \GlobExt_{\O_{\widetilde{Z}}}^2(K_{\widetilde{Z}}^{\vee}\otimes S,K_{\widetilde{Z}}\otimes S)^{\vee} . 
$$
As it is shown in \cite{Don-Mar}, its symplectic leaves are obtained by fixing the intersection of the support of the sheaf 
with the degeneracy divisor (\ref{eq:degeneracy-divisor}). 
Given that the curves in the Hilbert scheme satisfy conditions (\ref{prop:parabolicBNR21})--(\ref{prop:parabolicBNR25}) 
of Theorem \ref{thm:parabolicBNR}, the symplectic leaves are thus obtained by fixing the intersection of the 
support curve with the exceptional divisors $E_{i,j,j',\ldots ,j^{(m_i-1)}}$.

We will need a generalization of the setup of the previous paragraph to a relative situation. 
Namely, for the parameter space $\C^{r\sum_i (m_i+1)}$ of (\ref{eq:Dolbeault-moduli}) the product 
\begin{equation}\label{eq:CxZ}
   \C^{r\sum_i (m_i+1)} \times Z
\end{equation}
contains a tautological flat reduced subscheme $\mathcal{C}$ of relative dimension $0$ over $\C^{r\sum_i (m_i+1)}$ given by 
$$
  \mathcal{C} = \left( (\zeta_{i,j},\ldots ,\zeta_{i,j,j'\ldots , j^{(m_i)}}, \lambda_{i,j,j'\ldots , j^{(m_i)}})_{i,j,j'\ldots , j^{(m_i)}}, 
  \cup_i \C[[z_i]]/I_{\zeta_{i,j},\ldots ,\zeta_{i,j,j'\ldots , j^{(m_i-1)}}, \lambda_{i,j,j'\ldots , j^{(m_i)}} }
  \right)
$$
where 
$$
  I_{\zeta_{i,j},\ldots ,\zeta_{i,j,j'\ldots , j^{(m_i-1)}}}  = (z_i, \zeta_i - (\zeta_{i,j} + z_i \zeta_{i,j,j'} + 
    \cdots + z_i^{m_i-1} \zeta_{i,j,j'\ldots , j^{(m_i-1)}} + z_i^{m_i} \lambda_{i,j,j'\ldots , j^{(m_i)}}) )
$$
in the local chart of $Z$ over a neighborhood of $p_i$ with coordinates $z_i, \zeta_i$ (\ref{eq:zeta-i}). 
Let us denote by 
$$
  \widetilde{\mathcal{Z}} \to \C^{r\sum_i (m_i+1)} \times Z
$$
the blow-up of $\mathcal{C}$ in (\ref{eq:CxZ}). 
We are interested in the relative Hilbert scheme 
\begin{equation}\label{eq:relative-Hilbert}
   \Hilb_{\rel}(\widetilde{\mathcal{Z}}, \C^{r\sum_i (m_i+1)}, H) 
\end{equation}
of $\widetilde{\mathcal{Z}}$ with respect to $\C^{r\sum_i (m_i+1)}$, 
with Hilbert polynomial on the fibres equal to $H$ given in the previous paragraph. 
It admits a dense open subscheme 
\begin{equation}\label{eq:relative-Hilbert-Zariski}
   \Hilb_{\rel}^0(\widetilde{\mathcal{Z}}, \C^{r\sum_i (m_i+1)}, H) 
\end{equation}
parameterizing smooth connected curves over $\C^{r\sum_i (m_i+1)}$. 
Indeed, let $A\subset \C^{r\sum_i (m_i+1)}$ be the constructible subset consisting of 
nested sequences of eigenvalues for which the eigenvalues $\lambda_{i,j,j'\ldots , j^{(m_i)}}$ 
of the residue at $p_i$ restricted to common eigenspaces of the irregular part are of multiplicity $1$. 
Then, away from $A$ the curves in (\ref{eq:relative-Hilbert}) are smooth over $D_{\red}$, and 
in the fibers of (\ref{eq:relative-Hilbert}) over points of $\C^{r\sum_i (m_i+1)}\setminus A$ further Zariski 
open subsets parametrize curves that are everywhere smooth. 
Notice also that the dimensions $d_{i,j,j',\ldots ,j^{(m_i-1)}}$ may change in a discrete way giving different 
Hilbert polynomials, and this splits up the relative Hilbert scheme into several components. 
In the Dolbeault setup, this corresponds to letting the dimensions of the joint eigenspaces of the matrices in the irregular part vary. 

Finally, we are interested in 
\begin{equation}\label{eq:relative-compactified-Picard}
  \Pic_{\rel}(\widetilde{\mathcal{Z}}, \C^{r\sum_i (m_i+1)}, H,d)
\end{equation}
parameterizing torsion-free sheaves of given degree $d$ on the curves in (\ref{eq:relative-Hilbert}). 
It admits a dense open subscheme 
\begin{equation}\label{eq:relative-Picard}
  \Pic_{\rel}^0(\widetilde{\mathcal{Z}}, \C^{r\sum_i (m_i+1)}, H,d) \to  \Hilb_{\rel}^0(\widetilde{\mathcal{Z}}, \C^{r\sum_i (m_i+1)}, H) 
\end{equation}
parameterizing line bundles supported on a family of smooth connected curves over $\C^{r\sum_i (m_i+1)}$.

The Poisson bivector field $\Pi_{\Pic^0}$ on (\ref{eq:Picard}) described after Claim \ref{clm:Poisson} admits a 
straightforward extension to (\ref{eq:relative-Picard}) for which the natural map to $\C^{r\sum_i (m_i+1)}$ 
consists of Casimir operators, simply by pushing forward bivectors. Namely, the inclusion 
$$
  \widetilde{Z} \hookrightarrow \widetilde{\mathcal{Z}}
$$ 
gives rise to a map of cotangent bundles 
$$
  T^* \Pic_{\rel}^0(\widetilde{\mathcal{Z}}, \C^{r\sum_i (m_i+1)}, H,d) \to T^* \Pic_{\rel}^0(\widetilde{Z},H,d)
$$
and the Poisson structure on (\ref{eq:relative-Picard}) is then defined as the composition 
\begin{align*}
  T^* \Pic_{\rel}^0(\widetilde{\mathcal{Z}}, \C^{r\sum_i (m_i+1)}, H,d) & \wedge 
  T^* \Pic_{\rel}^0(\widetilde{\mathcal{Z}}, \C^{r\sum_i (m_i+1)}, H,d) \to  \\
  \to T^* \Pic_{\rel}^0(\widetilde{Z},H,d) & \wedge T^* \Pic_{\rel}^0(\widetilde{Z},H,d) \xrightarrow{\Pi_{\Pic^0}} \C. 
\end{align*}

Let us define 
$$
  \Mod_{\Dol}^{\irr,0} \subseteq \Mod_{\Dol}^{\irr}
$$
to consist of equivalence classes of Higgs bundles with smooth connected spectral curve $\widetilde{\Sigma}$, 
unramified over $D_{\red}$. With this notation the following result holds. 

\begin{thm}\label{thm:Poisson-isomorphism}
 The equivalence of categories of Theorem \ref{thm:parabolicBNR} induces a Poisson isomorphism between dense 
 open subets of the spaces (\ref{eq:Dolbeault-moduli}) and (\ref{eq:relative-compactified-Picard}): 
 \begin{align*}
  (\Mod_{\Dol}^{\irr,0},\Pi_{\Dol}) & \cong (\Pic_{\rel}^0(\widetilde{\mathcal{Z}}, \C^{r\sum_i (m_i+1)},H,d), \Pi_{\Pic}), \\
  d & = \delta + \frac{r(r-1)}2 \deg(L)
 \end{align*}
 with $\delta$ defined in Remark \ref{rk:dR-Dol-moduli}. 
\end{thm}

\begin{proof}
 First let us observe that by definition for $(\E,\theta) \in \Mod_{\Dol}^{\irr,0}$ the eigenspaces 
 of the residue of $\theta$ are $1$-dimensional. Therefore, the choice of a parabolic filtration on the 
 eigenspaces is vacuous, so Theorem \ref{thm:parabolicBNR} indeed identifies $\Mod_{\Dol}^{\irr,0}$ 
 with (\ref{eq:relative-Picard}). 
 
 It is now sufficient to show that the symplectic structures on the symplectic leaves get identified. 
 This is precisely the content of the key observation (24) of Proposition 5.1 of \cite{Sz-plancherel} 
 (c.f. also \cite{Ha-Hu} Proposition 2.30 in the case of the holomorphically trivial vector bundle 
 over the projective line). 
 The formula for $d$ follows directly from the well-known fact that the direct image of $\O_{\Sigma}$ is equal to 
 $$
  \O_C \oplus L^{-1} \oplus \cdots \oplus L^{1-r}. 
 $$
\end{proof}

\section{Acknowledgments}
This paper grew out of work carried out under the support of the Advanced Grant 
``Arithmetic and physics of Higgs moduli spaces'' no. 320593 of the European Research Council 
and the Lend\"ulet ``Low Dimensional Topology'' program of the Hungarian Academy of Sciences. 
The author would like to thank the \'Ecole Polytechnique F\'ed\'erale de Lausanne and 
the Alfr\'ed R\'enyi Institute of Mathematics for their hospitality.

\bibliography{BNR}
\bibliographystyle{plain}

\end{document}